\documentclass[12pt,draft]{amsart}
\usepackage[all]{xy}
\usepackage{amsfonts}
\usepackage{amssymb}

\usepackage[english]{babel}
\usepackage{ogonek}

\usepackage{enumerate}

\renewcommand{\le}{\leqslant}
\renewcommand{\ge}{\geqslant}

\newtheorem{teo}{Theorem}[section]
\newtheorem{lem}[teo]{Lemma}
\newtheorem{prop}[teo]{Proposition}

\newtheorem{cor}[teo]{Corollary}

\newtheorem{dfn}[teo]{Definition}
\newtheorem{rk}[teo]{Remark}
\newtheorem{ex}[teo]{Example}

\def\<{\langle}
\def\>{\rangle}

\def\a{\alpha}
\def\b{\beta}

\def\e{\varepsilon}

\def\r{\rho}

\def\s{\sigma}
\def\t{\tau}

\def\f{{\varphi}}

\def\G{{\Gamma}}

\def\C{{\mathbb C}}

\def\Z{{\mathbb Z}}

\def\cC{{\mathcal C}}
\def\cP{{\mathcal P}}
\def\cU{{\mathcal U}}

\def\End{\mathop{\rm End}\nolimits}

\def\Id{\operatorname{Id}}

\def\Tr{\operatorname{Tr}}

\def\1{\mathbf 1}

\newcommand{\wh}[1]{\widehat{#1}}

\def\Fix{\operatorname{Fix}}

\begin{document}

\title[Reidemeister type zeta functions]
{New zeta functions of Reidemeister type
and twisted Burnside-Frobenius theory}

\author{Alexander Fel'shtyn}
\thanks{The work of A.F. and M.Z. (Sections \ref{sec:prerem} and \ref{sec:dynrepzet})
is funded by the  Narodowe Centrum Nauki of Poland (NCN) (grant No.~\!2016/23/G/ST1/04280).}
\address{Instytut Matematyki, Uniwersytet Szczecinski,
ul. Wielkopolska 15, 70-451 Szczecin, Poland} \email{felshtyn@mpim-bonn.mpg.de,
 fels@wmf.univ.szczecin.pl}
 
\author{Evgenij Troitsky}
\thanks{The work of E.T. (Sections \ref{sec:compactreidnum},
 \ref{sec:finfindim}, \ref{sec:univversprof}, \ref{sec:calc}, and \ref{sec:examcounter}) is
supported by the Russian Science Foundation under grant 16-11-10018.}
\address{Dept. of Mech. and Math., Moscow State University,
119991 GSP-1  Moscow, Russia}
\email{troitsky@mech.math.msu.su}
\urladdr{
http://mech.math.msu.su/\~{}troitsky}

\author{Malwina Zi\k{e}tek}
\address{Instytut Matematyki, Uniwersytet Szczecinski,
ul. Wielkopolska 15, 70-451 Szczecin, Poland} \email{malwina.zietek@gmail.com}


\begin{abstract}
We introduce new zeta functions related to an endomorphism
$\phi$ of a discrete group $\G$. They are of two types: counting
numbers of fixed ($\rho\sim \rho\circ\phi^n$) 
irreducible representations for iterations of $\phi$ 
from an appropriate dual space of $\G$ and counting
Reidemeister numbers $R(\f^n)$ of different compactifications.
Many properties of these functions and their coefficients
are obtained. In many cases it is proved that these zeta functions
coincide.
The Gauss congruences are proved.
Useful asymptotic formulas for the zeta functions are found. 
Rationality is proved for some examples, which give also
the first counterexamples simultaneously for TBFT 
($R(\phi)$=the number of fixed irreducible unitary representations)
and TBFT$_f$ ($R(\phi)$=the number of fixed irreducible unitary 
finite-dimensional representations)
for an automorphism $\phi$ with $R(\phi)<\infty$.
\end{abstract}

\maketitle

\section*{Introduction}
Let $\phi$ be an endomorphism of a group $\G$.
The \emph{Reidemeister number} $R(\phi)$ of $\phi$
is the number of its \emph{Reidemeister} or
\emph{twisted conjugacy classes}
$$
\{g\}_\phi:=\{xg\phi(x^{-1}),\quad x\in\G\}.
$$

In this paper we will be mostly interested in discrete
groups (using the notation $\G$, $\phi$) and in compact
(Hausdorff) groups (using the notation $G$, $\f$).

The first group of problems related Reidemeister
numbers includes a study of validity of
the TBFT (twisted Burnside-Frobenius theorem (or theory))
for different classes of groups, a proof of the
Gauss congruences for the Reidemeister numbers
of iterations (mostly using the TBFT), and
a study of rationality of the corresponding
Reidemeister zeta function.  

The first formulation of TBFT, due to A.~Fel'shtyn and R.~Hill, 
says that the $R(\phi)$=the number of fixed irreducible unitary representations, if one of them is finite.
It was proved for automorphisms of abelian, compact, and abelian-by-finite
groups in \cite{FelHill,FelHillForum,feltroKT}. 
In \cite{FelTroVer} a counterexample to the TBFT
was detected and it led to a new form of the problem
TBFT$_f$, where it is conjectured that if $R(\phi)<\infty$,
then it coincides with the number of $\rho\sim\rho\circ\phi$,
where $\rho$ is a finite-dimensional irreducible representation.
The TBFT$_f$ was proved for polycyclic groups (for automorphisms
in \cite{polyc} and for endomorphisms in \cite{FelTroRJMP2018}).
In \cite{feltroKT} it was observed that TBFT$_f$ is not true
for some infinite groups with finite number of
ordinary conjugacy classes. 
In \cite{TroTwoEx} a counterexample to
TBFT$_f$ was detected among infinitely generated residually
finite groups. 
(We will give a common counterexample for TBFT and TBFT$_f$
in the present paper.) 
For related results we refer to \cite{ncrmkwb,TroTwoEx}.
Concerning rationality of the Reidemeister zeta function
we refer to the presentation in \cite{FelshB} and a recent
paper \cite{DekimpeTerBus2017ArxRation}.

The second group of problems, related to the first one, is
to determine groups with the $R_\infty$-\emph{property}
(each automorphism has infinite Reidemeister number).
The approaches differ for various classes under consideration
(e.g., branch groups, lattices, linear groups, solvable groups, etc.),
so in fact we have a splitted system of problems, rather one entire problem.
For endomorphisms the definition is not appropriate 
for a direct consideration,
because each group has an endomorphism $\psi$ with $R(\psi)=1$, namely 
$\psi:g\mapsto e$, for any $g$ (see the beginning of 
\cite{FelTroRJMP2018} for more information). 
Many papers were devoted to this problem recently.
This is not a main subject of the present paper and we only
refer to the following selection of papers on the problem
and bibliography therein 
\cite{ll,%
FelLeonTro,GK,gowon,%
MubeenaSankaran2014TrGr,%
Romankov,DekimpeGoncalves2014BLMS,FelNasy2016JGT,%
HaLee2015,FelTroJGT}.

We introduce here several interrelated zeta functions
determined by an endomorphism $\phi$ of a discrete group $\G$.
First we introduce \emph{dynamic representation theory}
zeta functions with $n^{th}$ coefficients equal
to the number of those irreducible unitary representations
(respectively, finite-dimensional, respectively, finite 
irreducible unitary representations) $\rho$ such that
$\rho\sim \rho\circ\phi^n$ (supposing these numbers to
be finite) (Definition \ref{dfn:represreidzeta}). 
We prove the Gauss congruences
for these coefficients, using the existence of an
appropriate dynamical system on a part of the unitary
dual with periodic points of period $n$ being exactly
the above mentioned representation classes  
(Theorem \ref{teo:congrue_rep_reide}).

Then we introduce a natural notion of an \emph{admissible
compactification} of $\G$ (Definition \ref{dfn:admcom})
with the profinite completion and the universal compactification
as main examples. We define the corresponding \emph{compactification
Reidemeister zeta function}.
Then we prove three statements about Reidemeister classes of
an endomorphism $\f$ of a compact Hausdorff group $G$:
1)calculation
of a finite Reidemeister number as the number of irreducible
representations $\rho$ such that
$\rho\sim \rho\circ\phi$ (Theorem \ref{teo:peterweylforendcomp}); 
2) description of those matrix
coefficients, which are constant on Reidemeister classes (Lemma \ref{lem:funczional}); and
3) proof of the TBFT in this situation (Corrolary \ref{cor:tbftcomp}).   

We prove that if $R(\f)<\infty$ the above mentioned representations
must be finite (not only finite-dimensional) (Theorem \ref{teo:onlyfinite}).
 
Using these facts, we prove under finiteness conditions
the coincidence of representation theory zeta-functions (related
to finite-dimensional and finite representations) and compactification
zeta-functions (for the profinite and the universal compactifications)
(Theorem \ref{teo:coinc_of_zetas}).  
 
We obtain an asymptotic formula for these zeta-functions
in Theorem \ref{teo:formulazeta} and Corollary \ref{cor:asimfor}.

Then we develop an example from \cite{TroTwoEx} to give a proof
of rationality in this situation of the zeta functions under consideration
(Proposition \ref{prop:calczetaex}). As a corollary of these calculations, we obtain the first
example of a group and its automorphism, giving a counterexample
to the TBFT and the TBFT$_f$ simultaneously (Theorem \ref{teo:counterex}).

\smallskip
The results of Sections \ref{sec:prerem} and \ref{sec:dynrepzet}
are obtained by A.F. and M.Z.

The results of Sections \ref{sec:compactreidnum},
 \ref{sec:finfindim}, \ref{sec:univversprof}, \ref{sec:calc}, and \ref{sec:examcounter}
 are obtained by E.T.

\smallskip 
\textsc{Acknowledgement:} 
Preliminary results on the subject were obtained by A.F. anf E.T. 
under hospitality of the Max-Planck-Institut for Mathematics in Spring 2017.

The authors are indebted to V.~M.~Manuilov and A.~I.~Shtern
for helpful discussions.
 
The work of A.F. and M.Z. (Sections \ref{sec:prerem} and \ref{sec:dynrepzet})
is funded by the  Narodowe Centrum Nauki of Poland (NCN) (grant No.~\!2016/23/G/ST1/04280).

The work of E.T. (Sections \ref{sec:compactreidnum},
 \ref{sec:finfindim}, \ref{sec:univversprof}, \ref{sec:calc}, and \ref{sec:examcounter})
is supported by the Russian Science Foundation under grant 16-11-10018.
 
\section{Preliminaries}\label{sec:prerem} 

We will need the following basic observation.  
\begin{lem}\label{lem:epimreid} 
Suppose, $\phi$ is a (continuous) endomorphism of $\G$ 
and $N$ is a normal $\phi$-invariant subgroup of $\G$.
Then the map $\G\to \G/N$ maps Reidemeister classes of
$\phi$ onto Reidemeister classes of the induced endomorphism
of $\G/N$.
\end{lem} 
 
Also we need the following results \cite{FelHillForum,FelshB} about
the Reidemeister classes of an
endomorphism $\f$ of a finite group $F$. 
Consider the action of $\f$ on usual (non-twisted) 
class functions (i.e. functions, which are constant on usual
conjugacy classes): $B: f\mapsto f\circ\f$. Using the
orbit-stabilizer theorem and the TBFT for finite groups
one can prove that $R(\f)=\Tr B=$the number of $\f$-fixed
usually conjugacy classes. This implies rationality of
the Reidemeister zeta function (see \cite[Theorem 17]{FelshB})
via the following calculation:
\begin{equation}\label{eq:razzetaforfin}
R_\f(z)=\exp \left(\sum_{n=1}^\infty \frac{R(\f^n)}{n}z^n\right)
=\exp \left(\sum_{n=1}^\infty \frac{\Tr B^n}{n}z^n\right)
=\det\exp \left(\sum_{n=1}^\infty \frac{B^n}{n}z^n\right)
\end{equation}
$$
=\det\exp(-\log(1-Bz))=\frac{1}{\det(1-Bz)}
$$
(at least for $|z|<1$). 
 
Suppose, $X$ is a set and $\sigma:X\to X$ is a mapping. 
Denote by $\Fix(\sigma^n)$ the set of fixed points of
$\sigma^n$, i.e., $n$-\emph{periodic points of} $\sigma$.
Suppose, its cardinality $|\Fix(\sigma^n)|<\infty$ for any $n$.
Then the \emph{Artin--Mazur zeta function} is defined by
$$
AM_\sigma(z)=\exp \left(\sum_{n=1}^\infty \frac{|\Fix(\sigma^n)|}{n}z^n\right).
$$
It is well known (see, e.g. \cite{Baladi}) that
using the Taylor series of $\log(1-z^p)$ the function 
$AM_\sigma(z)$ can be written as a formal
Euler product over all \emph{primitive} 
periodic orbits $\t=\{x_\t,\s(x_\t),\dots,\s^{p-1}(x_\t)$, 
with $p=p(\t)$, such that $\s^p(x_\t)=x_\t$ and
$\s^k(x_\t)\ne x_\t$  for $0<k<p$. Namely,
\begin{equation}\label{eq:baladi}
AM_\sigma(z)=\prod_{\t\mathrm{\: primitive\: periodic\: orbit}}
\frac{1}{1-z^{p(\t)}}.
\end{equation}
The elements of a primitive periodic orbit of length $p$
 are called $p$-\emph{periodic} elements.
If $|X|<\infty$ this implies \emph{rationality} of $AM_\sigma(z)$.
In the case of a bijective $\sigma$ and finite $X$ 
the following \emph{functional equation} is well-known in the field (see e.g. \cite{Yoshida}) 
\begin{equation}\label{eq:yoshida_fe}
AM_\sigma(1/z)=(-z)^{|X|} \det(\s_C) AM_\sigma(z),
\end{equation}
where $\s_C:C(X)\to C(X)$ is the induced linear mapping,
$C(X)\cong \C^{|X|}$. One
can generalize (\ref{eq:yoshida_fe}) to the non-bijective
case in the following way (see e.g. \cite[Sect. 2.3.2]{FelshB}) 
\begin{multline}\label{eq:fe_endo}
AM_\sigma(1/z)=\prod_\t \frac{1}{1-z^{-p(\t)}} 
=\prod_\t \frac{-z^{p(\t)}}{1-z^{p(\t)}}
\\
=\prod_\t (-z^{p(\t)}) \prod_\t \frac{1}{1-z^{p(\t)}}
=(-1)^a z^b AM_\sigma(z),
\end{multline}
where $a$ is the number of primitive orbits and
$b$ is the number of periodic elements. If $\s$ is a bijection,
each point is periodic and $b=|X|$, while 
$$
\det(\s_C)=\prod_\t (-1)^{p(\t)-1}= \prod_\t (-1)^{p(\t)}\cdot
(-1)^a=(-1)^{|X|}\cdot (-1)^a 
$$
and we arrive to (\ref{eq:yoshida_fe}).
 
\section{Dynamic representation theory zeta functions}\label{sec:dynrepzet} 
Suppose, $\phi$ is an endomorphism of a discrete group $\G$.
Generally the correspondence $\widehat{\phi}:\rho\mapsto \rho\circ\phi$
does not define a dynamical system (an action of the semigroup
of positive integers) on the unitary dual $\widehat{\G}$    
or its finite-dimensional $\widehat{\G}_f$ part, or finite 
$\widehat{\G}_{ff}$ part, because in contrast with the
authomorphism case, the representation $\rho\circ\phi$
may be reducible. Here the
\emph{unitary dual} is the space of equivalence classes of
unitary irreducible representations of $\G$, equipped with the
\emph{hull-kernel} topology, $\widehat{\G}_f$ is its subspase
formed by finite-dimensional representations, and  
$\widehat{\G}_{ff}$ is formed by \emph{finite} representations.

Nevertheless we can consider representations $\rho$ such that
$\rho \sim \rho\circ\phi$.

\begin{dfn}\label{dfn:represreidnum}
{\rm
A \emph{representation theory Reidemeister number}
$RT(\phi)$
is the number of all $[\rho]\in \widehat{\G}$ such that
$\rho \sim \rho\circ\phi$. Taking $[\rho]\in \widehat{\G}_f$
(respectively $[\rho]\in \widehat{\G}_{ff}$) we obtain
$RT^f(\phi)$ (respectively $RT^{ff}(\phi)$). Evidently
$RT(\phi)\ge RT^f(\phi)\ge RT^{ff}(\phi)$.
}
\end{dfn}

\begin{dfn}\label{dfn:represreidzeta}
{\rm
If these numbers are finite for all powers of $\phi$,
we define the corresponding \emph{dynamic representation
theory zeta functions}
$$
RT_\phi(z):=\exp \left(\sum_{n=1}^\infty \frac{RT(\phi^n)}{n}z^n\right),\quad
RT^f_\phi(z):=\exp \left(\sum_{n=1}^\infty \frac{RT^f(\phi^n)}{n}z^n\right),
$$
$$
RT^{ff}_\phi(z):=\exp \left(\sum_{n=1}^\infty \frac{RT^{ff}(\phi^n)}{n}z^n\right).
$$
}
\end{dfn}

The importance of these numbers is justified by the
following dynamical interpretation.
In \cite{FelTroRJMP2018} the following
``dynamical part'' of the dual space, where $\widehat{\phi}$ and all
its iterations $\widehat{\phi}^n$  define a dynamical system,
was defined. 

\begin{dfn}[Def. 2.1 in \cite{FelTroRJMP2018}]
{\rm
A class $[\rho]$ is called
a $\wh\phi$-\textbf{f}-point,
if $\rho\sim \rho\circ\phi$ (so, these are the points
under consideration in above definitions).
}
\end{dfn}

\begin{dfn}[Def. 2.2 in \cite{FelTroRJMP2018}]
{\rm
An element $[\rho]\in \wh{\G}$ (respectively, in $\wh{\G}_f$
or $\wh{\G}_{ff}$)
is called $\phi$-\emph{irreducible} if $\r\circ \phi^n$ is
irreducible for any $n=0,1,2,\dots$.

Denote the corresponding subspaces of $\wh{\G}$ (resp., $\wh{\G}_f$
or $\wh{\G}_{ff}$)
by $\wh{\G}^\phi$ (resp., $\wh{\G}^\phi_f$ or $\wh{\G}^\phi_{ff}$). 
}
\end{dfn}

\begin{lem}[Lemma 2.4 in \cite{FelTroRJMP2018}]\label{lem:keylem}
Suppose, the representations $\rho$ and $\rho\circ\phi^n$ are
equivalent  for some $n\geq 1$. Then $[\rho]\in \wh \G^\phi$.
\end{lem}

\begin{cor}[Corollary 2.5 in \cite{FelTroRJMP2018}]\label{cor:periodicanddyn}
Generally, there is no dynamical system defined by $\widehat{\phi}$
on $\wh \G$ (resp., $\wh \G_f$, or $\wh{\G}_{ff}$).
We have only the well-defined notion of a
$\wh\phi^n$-\textbf{f}-point.

A well-defined dynamical system exists on 
$\wh \G^\phi$ (resp, $\wh \G^\phi_f$, or $\wh{\G}^\phi_{ff}$). 
Its $n$-periodic points are exactly $\wh\phi^n$-\textbf{f}-points.
\end{cor}

Let us remark the number of $\wh\phi^n$-\textbf{f}-points 
was denoted in \cite{FelTroRJMP2018} by $\mathbf{F}(\wh\phi^n)$, but here we denote
it more conceptually by $RT(\phi^n)$.
We refer to \cite{FelTroRJMP2018} for proofs and
details.

Once we have identified
the coefficients of representation theory zeta functions 
with the
numbers of periodic points of a dynamical system, 
the standard argument with the
M\"obius inversion formula
(see e.g. \cite[p.~104]{FelshB}, \cite{FelTroRJMP2018}) 
gives the following statement.

\begin{teo}\label{teo:congrue_rep_reide}
Suppose, $RT(\phi^n)<\infty$ for any $n$.
Then we have the following Gauss congruences
for representation theory Reidemeister numbers:
 $$
 \sum_{d\mid n} \mu(d)\cdot RT(\phi^{n/d}) \equiv 0 \mod n
 $$
for any $n$.

A similar statement is true for $RT^f(\phi^n)$ and 
$RT^{ff}(\phi^n)$.
\end{teo}

Here the above \emph{M\"obius function} is defined as
$$
\mu(d) =
\left\{
\begin{array}{ll}
1 & {\rm if}\ d=1,  \\
(-1)^k & {\rm if}\ d\ {\rm is\ a\ product\ of}\ k\ {\rm distinct\ primes,}\\
0 & {\rm if}\ d\ {\rm is\ not\ square-free.}
\end{array}
\right.
$$

The following statement evidently follows from the 
definitions.

\begin{prop}\label{prop:tbftimplcoin}
Suppose, $\phi:\G\to\G$ is an endomorphism and $R(\phi)<\infty$.
If TBFT $($resp., TBFT$_f)$ is true for $\G$ and $\phi$,
then $R(\phi)=RT(\phi)$ $($resp, $R(\phi)=RT^f(\phi)=RT^{ff}(\phi))$.

If the suppositions keep for $\phi^n$, for any $n$, then
$R_\phi(z)=RT_\phi(z)$ $($resp., 
$R_\phi(z)=RT^f_\phi(z)=RT^{ff}_\phi(z))$.
\end{prop}

\begin{proof}
The only non-trivial fact is that $RT^f(\phi)=RT^{ff}(\phi)$
under our suppositions. We postpone this till a more general
consideration in Theorem \ref{teo:coinc_of_rei_num}.
\end{proof}

\begin{teo}\label{teo:rationalold}
Suppose, TBFT $($resp., TBFT$_f)$ is true for $\G$ and $\phi^n$;
and $R(\phi^n)<\infty$ for any $n$.
If $R_\phi(z)$ is rational, then 
$RT_\phi(z)$ $($resp., 
$RT^f_\phi(z)=RT^{ff}_\phi(z))$ is rational.

In particular, $RT^f_\phi(z)=RT^{ff}_\phi(z)$ is rational
in the following cases:

\begin{itemize}
\item [1.] $\G$ is a finitely generated abelian group;
\item [2.] $\G$ is a finitely generated torsion free
nilpotent group;
\item [3.] $\G$ is a crystallographic group 
with diagonal holonomy $\Z_2$ and $\phi$ is an automorphism.
\end{itemize}
\end{teo}

\begin{proof}
The first part follows immediately from Proposition \ref{prop:tbftimplcoin}.

In the second case we have a polycyclic group
and in the third case we have an almost polycyclic
group.
TBFT$_f$ for the cases under consideration is proved
in \cite{FelHill} (for endomorphisms of
abelian groups) (see also \cite{feltroKT}),
in \cite{polyc} (for automorphisms of almost polycyclic
groups) and in \cite{FelTroRJMP2018} (for endomorphisms of  polycyclic groups).

Rationality of $R_\phi(z)$ is proved in \cite{FelHillForum}
for the first and second cases (see \cite{FelshB}) 
and for the third case in \cite{DekimpeTerBus2017ArxRation}. 
\end{proof}

\begin{rk}\label{rk:lefshetzforabelian}
{\rm
In the same way with the help of Proposition 
\ref{prop:tbftimplcoin} one can extract from \cite{FelshB}
more information in the first case above. Namely,
in this case all irreducible representations are $1$-dimensional
and $RT_\phi(z)=RT^f_\phi(z)=RT^{ff}_\phi(z))$.
Let $\G_T$ be the characteristic torsion subgroup with
$\Z^n=\G/\G_T$. Denote by $\phi_T:\G_T\to\G_T$ and
by $\phi':\Z^n \to \Z^n$ the induced endomorphisms.
Then 
$$
RT_\phi(z)=L_{\hat{\phi}}(\sigma z)^{(-1)^r},
$$
where $L$ is the Lefschetz zeta function of 
$\widehat{\phi}:\widehat{\G}\to\widehat{\G}$, 
$\sigma=(-1)^p$ where $p$ is the number of real eigenvalues
of the linear operator on the Lie algebra of $\widehat{G}$,
corresponding to $\widehat{\phi'}$
such that $\lambda<-1$
and $r$ is the number of its real eigenvalues
$\lambda$
such that $|\lambda | > 1$ (see \cite[Theorems 28 and 29]{FelshB} for details).
}
\end{rk}

Since, by the definition, 
$RT_\phi(z)$ (resp, $RT^f_\phi(z)$, or $RT^{ff}_\phi(z)$)
is the Artin-Mazur zeta function of $\widehat{\phi}$ on
$\wh \G^\phi$ (resp, $\wh \G^\phi_f$, or $\wh{\G}^\phi_{ff}$),
we obtain immediately from (\ref{eq:baladi}) the following

\begin{prop}\label{prop:baladi}
$$
RT_\phi(z)=\prod_{\t\mathrm{\: primitive\: periodic\: orbit}}
\frac{1}{1-z^{p(\t)}}
$$
and similarly for $RT^f_\phi(z)$ and $RT^{ff}_\phi(z)$.
\end{prop}

Now we shall prove rationality of zeta functions and
related facts under
restriction on $\widehat{\G}$ rather than on $ \G$ itself,
as above. 
More precisely, we will suppose that
$\wh \G^\phi$ (resp, $\wh \G^\phi_f$, or $\wh{\G}^\phi_{ff}$)
is finite. In this case 
$RT_\phi(z)$ (resp, $RT^f_\phi(z)$, or $RT^{ff}_\phi(z)$)
is the Artin-Mazur function on a finite set
and from Proposition \ref{prop:baladi} and (\ref{eq:fe_endo})
we obtain the following statement.

\begin{teo}\label{teo:onfinite}
Suppose that
$\wh \G^\phi$ (resp, $\wh \G^\phi_f$, or $\wh{\G}^\phi_{ff}$)
is finite. Then $RT_\phi(z)$ (resp, $RT^f_\phi(z)$, or $RT^{ff}_\phi(z))$
is rational and satisfies the following functional equation
$$
RT_\phi(1/z)= (-1)^a z^b RT_\phi(z),
$$
where $a$ is the number of primitive orbits and
$b$ is the number of periodic elements of $\widehat{\phi}$
on $\wh \G^\phi$. Similarly for 
$RT^f_\phi(z)$ and $RT^{ff}_\phi(z)$.
\end{teo}

\begin{ex}\label{ex:finite}
{\rm
An evident example is a finite group.
}
\end{ex}

\begin{ex}\label{ex:osinetc}
{\rm
Less evident examples give some infinite groups
with finitely many conjugacy classes:
Osin group \cite{Osin}, Ivanov group  
and some HNN extensions  
described in \cite{polyc}.
Then $R(\Id)$ is finite. These groups have only one
finite-dimensional representation (the trivial one)
and thus enter conditions of Theorem \ref{teo:onfinite}
(see \cite{polyc} for necessary proofs).
}
\end{ex}

\section{Compactifications and Reidemeister numbers}\label{sec:compactreidnum}

\begin{dfn}\label{dfn:phiclassfunct}
{\rm
A $\phi$-\emph{class function} is a function being
constant on Reidemeister classes.
}\end{dfn}

\begin{lem}\label{lem:clopeniffin}
If $G$ is compact Hausdorff and $R(\f)<\infty$, then Reidemeister
classes are clopen (closed and open). In particular,
$\f$-class functions are continuous.
\end{lem}

\begin{proof}
Reidemeister classes are orbits of the 
continuous twisted action
$g\mapsto xg\f(x^{-1})$ of $G$ on itself.
Thus, they are compact. Hence, closed, because $G$
is Hausdorff. The complement to each of them is
a finite union of closed sets. 
\end{proof}

\begin{dfn}\label{dfn:admcom}{\rm
A \emph{compactification} $\cC$ of a group $G$
is a couple consisting  of a compact Hausdorff
group $\cC(G)$ and a (continuous) homomorphism
$\alpha_\cC:G\to \cC(G)$ with dense image.
We say that $\cC$ is \emph{admissible for 
an endomorphism} $\phi:G\to G$ if 
\begin{itemize}
\item $\alpha_\cC$ has a $\phi$-invariant kernel (may be 
non-trivial);
\item the induced homomorphism $\cC(\phi):\cC(G)\to\cC(G)$ is continuous.
\end{itemize}
If $\cC$ is admissible for any endomorphism $\phi:G\to G$,
we say that it is \emph{admissible}.
The induced homomorphism  $\cC(\phi)$
is the \emph{compactification of} $\phi$.
Denote by $R_\cC(\phi)$ the corresponding 
Reidemeister numbers, i.e. $R_\cC(\phi)=R(\cC(\phi))$.

The corresponding \emph{compactification 
Reidemeister zeta function} is defined as
$$
R^\cC_\phi(z):=\exp \left(\sum_{n=1}^\infty \frac{R_\cC(\phi^n)}{n}z^n\right)
$$
supposing that all $R_\cC(\phi^n)$ are finite.
}
\end{dfn}

\begin{ex}{\rm
The main examples of admissible compactifications
come from
the \emph{profinite completion} $\cP$ and the
\emph{universal compactification} $\cU$ (see Lemma
\ref{lem:prof_and_uni_are_admiss}).
}\end{ex}

The \emph{universal compactification} $\cU(\G)$ is defined as
the closure of the image of the
diagonal homomorphism from $\G$ to the topological
(Tikhonoff) product of unitary groups being the ranges
of all (equivalence classes of) finite dimensional
irreducible unitary representations of $\G$. 
Denote by $\alpha_\cU$ the natural map $\alpha_\cU:\G\to \cU(\G)$. 
The pair $(\cU(\G),\alpha_\cU)$
enjoys the following \emph{universal property}:
for any homomorphism $\alpha':\G\to G$, where $G$ is some
compact group there exists a unique continuous
homomorphism $\gamma$ such that the diagram
\begin{equation}\label{eq:univproperty}
\xymatrix{
\G \ar[dr]_{\alpha'} \ar[r]^{\alpha_\cU } & \cU(\G) \ar[d]^{\gamma} \\
& G
}
\end{equation}
commutes
and is uniquely defined by it
(see, e.g. \cite[\S 16]{DixmierEng} for details).
 
The \emph{profinite completion} $\alpha_\cP:\G\to \cP(\G)$ 
is the natural homomorphism to 
the closure $\cP(\G)$
of the image of the diagonal homomorphism
to the Tikhonoff product of all finite quotients of $\G$.
It enjoys a universality property, similar to 
(\ref{eq:univproperty}), but with a \emph{profinite}
(i.e. Hausdorff, compact, and totally disconnected) 
group $G$ instead of a general compact group $G$ (see e.g. 
\cite{RibesZalesskiiBook,WilsonBook}
for details).
The profinite completion also can be defined in the
same way as the universal completion: namely we need
to take all finite representations instead of all
finite-dimensional representations.
The equivalence follows now from the universal property
of profinite completion and the decomposition of the
regular representation of a finite group.

\begin{lem}\label{lem:prof_and_uni_are_admiss}
The profinite completion $\cP$ and the
universal compactification $\cU$ are admissible
compactifications.
\end{lem}

\begin{proof}
This follows immediately from the universal properties.
Indeed, a continuous homomorphism
$\phi'$ making the diagram
$$
\xymatrix{
\G \ar[d]_\phi \ar[r]^{\alpha_\cU}  & \cU(\G) \ar[d]^{\phi'} \\
\G \ar[r]_{\alpha_\cU} &  \cU(\G)
}
$$
commutative can exist if and only if both properties
in the definition of admissibility are fulfilled.
Similarly, for the profinite completion.
\end{proof}

\begin{lem}\label{lem:firstestim}
For any admissible $\cC$, we have 
$R_\cC (\phi) \le R(\phi)$.
\end{lem}

\begin{proof}
Consider the image $\alpha_\cC(\G)$ of the compactification
homomorphism $\alpha_\cC:\G\to\cC(\G)$.
Then $\alpha_\cC:\G\to\alpha_\cC(\G)$ is an epimorphism, and
Lemma \ref{lem:epimreid} implies $R(\phi)\ge R(\cC(\phi)|_{\alpha_\cC(\G)})$. 
Evidently, any Reidemeister class in $\alpha_\cC(\G)$ is inside some class of $\cC(\phi)$, 
which is a compact
set being an orbit of the twisted action of the compact group
$\cC(\G)$.
Hence, 
the closure of a class in $\alpha_\cC(\G)$ still is inside a class
 of $\cC(\phi)$. On the other hand, the density of $\alpha_\cC(\G)$
 implies that each class of $\cC(\G)$ contain a class of
$\alpha_\cC(\G)$.
\end{proof}

From the universal property (\ref{eq:univproperty})
we obtain an epimorphism $\cU(\G)\to \cP(\G)$ 
and Lemma \ref{lem:epimreid} immediately implies the following
statement.

\begin{lem}\label{lem:plessu}
$R_\cU(\phi)\ge R_\cP(\phi)$.
\end{lem}

Now we need to generalize Theorem 4.3 and Lemma 5.1 
from \cite{FelTroRJMP2018}
from finite to compact groups in the following way.

\begin{teo}[cf. \cite{FelHill}]\label{teo:peterweylforendcomp}
Let $\f:G\to G$ be an endomorphism of a compact 
Hausdorff group $G$.
Suppose, the Reidemeister number $R(\f)<\infty$.
Then  $R(\f)$ 
coincides with
the number of $\wh\f$-{\rm\textbf{f}}-points on $\wh G$.
\end{teo}

\begin{proof}
Let us note that $R(\f)$ is equal to the dimension of
the space of $\f$-class functions (i.e. those functions
that are constant on Reidemeister classes). 
By Lemma \ref{lem:clopeniffin} these functions
are continuous. 
Thus, they can be described as fixed elements of the action
$a\mapsto g a \f(g^{-1})$ on the group algebra $C^*(G)$.
For the latter algebra we have the Peter-Weyl decomposition
$$
C^*(G)\cong \bigoplus_{[\rho]\in \wh G} \End V_\rho,\qquad
\rho:G\to U(V_\rho),
$$
which respects the left and right $G$-actions and the right-hand
side is equipped with the sup-norm. Hence,
$$
R(\f)=\sum\limits_{[\rho]\in \wh G} \dim T_\rho,
\qquad T_\rho:=\{a\in \End V_\rho\: |\: a=\rho(g) a \rho(\f(g^{-1})
\mbox{ for all } g\in G\},
$$
where we allow at this stage infinitely many non-zero summands.
Thus, if $0\ne a \in T_\rho$, then
$a$ is an intertwining operator between the irreducible
representation $\rho$ and some representation $\rho\circ\f$.
This implies that $\rho$ is equivalent to some 
(irreducible) subrepresentation $\pi$ of $\rho\circ\f$
(cf. \cite[VI, p.57]{NaimarkStern}). Hence, 
$\dim \rho=\dim \pi$, while $\dim \rho = \dim \rho\circ\f$.
Thus, $\pi=\rho\circ\f$, and is irreducible. 
In this situation $\dim T_\rho=1$ by the Schur lemma.
Evidently, vice versa, if $\rho\sim \rho\circ \f$ then 
$\dim T_\rho =1$.
Hence,
$$
R(\f)=\sum\limits_{[\rho]\in \wh G} \left\{
\begin{array}{l}
1,\mbox{ if }\rho\sim \rho\circ \f \\
0,\mbox{ if }\rho\not\sim \rho\circ \f
\end{array}
\right.
= \mbox{ number of }\wh\f\mbox{-\textbf{f}-points}.
$$
\end{proof}

\begin{lem}\label{lem:funczional}
Let $\rho$ be a (finite-dimensional) 
irreducible representation of a compact group $G$. 
It is a $\wh\f$-\textbf{f}-point 
of an endomorphism $\f:G\to G$, if and only if
there exists a non-zero $\f$ class function
being a matrix coefficient of $\rho$.

In this situation this function is unique
up to scaling and is defined by the formula
\begin{equation}\label{eq:formulaclassfun}
T_{S,\rho}:g\mapsto \Tr (S\circ \rho(g)),
\end{equation}
where $S$ is an intertwining operator between
$\rho$ and $\rho\circ \f$:
$$
\rho(\f(x)) S= S \rho(x)\quad\mbox{ for any }x\in G.
$$

In particular, 
TBFT is true for $\f$ if and only 
if the above matrix coefficients form a base of
the space of $\f$-class functions.
\end{lem}

\begin{proof}
First, let us note that (\ref{eq:formulaclassfun})
defines a class function:
$$
T_{S,\rho}(x g \f(x^{-1}))= 
\Tr (S \rho(x g \f(x^{-1})))=
\Tr (\rho(\f(x)) S \rho(g) \rho(\f(x^{-1}))=\Tr (S\rho(g)).
$$
If $S\ne 0$, then $\rho(a)=S^*$ for some $a\in C^*(G)$,
and $\Tr(SS^*) \neq 0$. Thus, the $\f$-class function is non-zero.
On the other hand, any non-trivial
matrix coefficient of $\rho$, i.e.
a functional $T:\End(V_\rho)\to\C$, has the form
$g\mapsto \Tr(D\rho(g))$ for some fixed matrix $D\ne 0$.
If it is a $\f$-class function, then for any $g\in G$,
or similarly, $a\in C^*(G)$,
$$
\Tr(D\rho(a))=\Tr(D \rho(x a \f(x^{-1})))=
\Tr(\rho(\f(x^{-1})) D \rho(x) \rho(a) ).
$$
Since $\rho(a)$ runs over the entire matrix algebra,
this implies 
$$
D=\rho(\f(x^{-1})) D \rho(x),\qquad \mbox{or}\qquad
\rho(\f(x)) D= D \rho(x),
$$
 i.e., $D$ is the desired
non-zero intertwining operator.

The uniqueness up to scaling follows now from 
the explicit formula and the Shur lemma.

The last statement follows from linear
independence of matrix coefficients of
non-equivalent representations.
\end{proof}

Combining these two results one obtains
\begin{cor}\label{cor:tbftcomp}
TBFT is true for endomorphisms of compact Hausdorff groups.
\end{cor}

As it was explained above, this implies
\begin{cor}\label{cor:congrcompactif}
If $R_\cC(\phi^n)<\infty$ for any $n$, they satisfy
the Gauss congruences.
\end{cor}

\section{Finite representations versus finite-dimensional ones}\label{sec:finfindim}
We will extend in this section a result from \cite{FelTroJGT}
to compact groups and endomorphisms.

\begin{teo}\label{teo:onlyfinite}
Suppose, $\phi:G\to G$ is an endomorphism of a compact
group $G$ and $R(\phi)<\infty$. If a matrix coefficient of
some finite-dimensional irreducible representation $\rho$ of $G$ is
a $\phi$-class function, then $\rho$ is finite.
\end{teo}

\begin{proof}
We give only a sketch of the proof, because it is more or
less the same as in \cite{FelTroJGT}.
Suppose, $f_\r$ is a non-trivial matrix coefficient. Hence,
its left translations, being once again matrix coefficients
of $\r$, generate a translation invariant subspace $W$ of 
the finite-dimensional space $V_\r\otimes V_\r^*$.
Hence, $W$ is a space of a finite-dimensional representation, which is 
isomorphic to a direct sum of several 
copies of $\r$ (see, e.g. 
\cite[Ch. IV]{NaimarkStern}). The space $W$
has a basis $L_{g_1} f_\r,\dots,L_{g_k} f_\r$.
Thus, all functions from $W$ take only finitely many
values (with level sets of the form $\cap_i g_i U_j$, where
$U_j$ are the level sets of $f_\r$). 
Thus, there exists a finite partition
$G=V_1\sqcup\dots\sqcup V_m$ such that elements of $W$ are
constant on the elements of the partition and for each
pair $V_i \ne V_j$ there exists a function from $W$ taking distinct
values on them. Note that these sets $V_i$ are closed and open,
because matrix coefficients are continuous. 
Hence, left translations map $V_i$ onto
each other.  This means that the representation $G$ on $W$ factorizes
through (a subgroup of) the permutation group on $m$
elements, i.e. a finite group. The same is true for its
subrepresentation $\r$, thus it is finite by definition.
\end{proof}

The following statement was known for automorphisms
of locally connected compact groups (see \cite{feltroKT}),
while here our main subject is the class of totally disconnected
groups.

\begin{cor}\label{cor:calccomp}
If $\f:G\to G$ is an endomorphism of a compact group $G$
with $R(\f)<\infty$, then there exists a closed normal
$\f$-invariant
subgroup of finite index $G_0$ such that the epimorphism
$$
p: G\to G/G_0=:F
$$
onto a (finite) group $F$ gives a bijection of Reidemeister
classes.
\end{cor}

\begin{proof}
The intersection of group kernels of all finite representations $\rho$
such that $\rho\sim\rho\circ\f$ can be taken as $G_0$ by
Theorem \ref{teo:peterweylforendcomp}, Lemma \ref{lem:funczional}, and
Theorem \ref{teo:onlyfinite}.
\end{proof}

\section{Comparing zeta functions}\label{sec:univversprof}

Now we will use Theorem \ref{teo:onlyfinite} to prove the
following statement.
\begin{teo}\label{teo:Rp=Ru=RT}
We have for an endomprphism $\phi:\Gamma\to\Gamma$,
\begin{itemize}
\item[1)] $R_\cU(\phi)=RT^f(\phi)$;
\item[2)] $R_\cU(\phi)=R_\cP(\phi)$ if $R_\cU(\phi)$ is finite;
\item[3)] $R_\cP(\phi)=RT^{ff}(\phi)$ if $R_\cP(\phi)$ is finite.
\end{itemize}
\end{teo}

\begin{proof}
The first statement follows immediately from 
Theorem \ref{teo:peterweylforendcomp}.

The second statement follows from the same theorem
in combination with Lemma \ref {lem:funczional}
and Theorem \ref{teo:onlyfinite}.

For the third statement observe that, by 
Theorem \ref{teo:peterweylforendcomp}, $R_\cP(\phi)$
is equal to the number of (finite-dimensional)
irreducible representations $\rho$ of $\cP(\G)$,
such that $\rho\sim \rho\circ \phi$. 
Since $R_\cP(\phi)<\infty$,
by Theorem \ref{teo:onlyfinite}, these representations are finite.
But finite representations of the profinite completion 
are just finite representation of $\G$ itself (cf.
the argument before Lemma \ref{lem:prof_and_uni_are_admiss}).
\end{proof}

\begin{rk}\label{rk:nikolovsegal}
{\rm
It seems possible to make the conditions weaker
in the case of finitely generated groups using
approaches from \cite{NikolovSegal2007AnnMat1}.
}
\end{rk}

\begin{rk}\label{rk:congranother}
{\rm
Combining Corollary \ref{cor:congrcompactif}
and Theorem \ref{teo:Rp=Ru=RT} one obtains another proof
of Theorem \ref{teo:congrue_rep_reide}.
}
\end{rk}

\begin{teo}\label{teo:coinc_of_rei_num}
If $R_\cU(\phi)<\infty$, then $R_\cU(\phi)=R_\cP(\phi)=
RT^{f}(\phi)=RT^{ff}(\phi)$.
In particular, these equalities are true, if $R(\phi)<\infty$.
\end{teo}

\begin{proof}
The main statement is a combination of items from
Theorem \ref{teo:Rp=Ru=RT}. The other statement follows
from Lemma \ref{lem:firstestim}.
\end{proof}

This statement implies immediately

\begin{teo}\label{teo:coinc_of_zetas}
If $R_\cU(\phi^n)<\infty$, $n=1,2,\dots$,
then $R^\cU_\phi(z)=R^\cP_\phi(z)=
RT^{f}_\phi(z)=RT^{ff}_\phi(z)$.
In particular, these equalities are true, if $R(\phi^n)<\infty$
for all $n$.
\end{teo}

\begin{rk}\label{rk:details}
{\rm
One can immediately obtain several equalities 
for zeta functions under more weak conditions for powers,
similar to those from Theorem \ref{teo:Rp=Ru=RT}.
}
\end{rk}

\section{Calculating zeta function}\label{sec:calc}

We will need the following evident observation. Consider two power
series $\sum_{k=1}^\infty a_k z^k$ and 
$\sum_{k=1}^\infty b_k z^k$ with non-negative coefficients:
$a_k\ge 0$, $b_k\ge 0$. Suppose, the first series is 
uniformly convergent on a closed disk $D$ of radius $d$
and $b_k\le a_k$ for all $k$. Denote $z_0=d$.
Then we have the following estimation, for $z\in D$:
\begin{equation}\label{eq:eatimforseries}
\left|\sum_{k=s}^p b_k z^k\right|\le
\sum_{k=s}^p b_k |z|^k \le \sum_{k=s}^p a_k d^k. 
\end{equation}
In particular, the second series is uniformly convergent on $D$.

Let $R(\f^n)<\infty$ for any $n$, where $\f:G\to G$
is an endomorphism of a compact group $G$. In accordance with
the above results (Corollary \ref{cor:calccomp}), there is
a countable collection of normal $\phi$-invariant subgroups $G_i$
of $G$ of finite index such that $p_i: G\to G/G_i=:F_i$ gives a bijection of Reidemeister numbers.

\begin{teo}\label{teo:formulazeta}
Suppose, $R(\f^n)<\infty$ for any $n$, where $\f:G\to G$
is an endomorphism of a compact group $G$. 
Suppose, $\sum_k \frac{R(\f^k)}k z^k$ is uniformly convergent
on some closed disk $D$  of radius $d<1$.
Choose an above defined collection of subgroups, and let
$B_i$ denote the operator $B$ from
Section \ref{sec:prerem} on class functions on $F_i$.
Then 
$$
R_\f(z)=\lim_{i\to \infty} \frac{1}{\det(1-B_i z)},\qquad z\in D.
$$
\end{teo}

\begin{proof}
Take $z\in D$ and arbitrary $\e>0$.
Denote $u=\sum_k \frac{R(\f^k)}k z^k$ and choose $\delta>0$
such that $|\exp(u)-\exp(u')<\e$, if $|u-u'|<\delta$.
Choose $k_0$ so large that   
\begin{equation}\label{eq:ozenkahvo}
\left|\sum_{k=k_0}^\infty \frac{R(\f^k)}k d^k\right|
=\sum_{k=k_0}^\infty \frac{R(\f^k)}k d^k < \frac{\delta}{2}.
\end{equation}
Now find a sufficiently large $i_0$ such that
\begin{equation}\label{eq:sovpadcfin}
R(\f^k)=R((\f_{i})^k), \qquad k=1,\dots k_0-1,\qquad i=i_0, i_0+1,
\dots.
\end{equation}
For $k=k_0, k_0+1, \dots$, we have $R(\f^k)\ge R((\f_{i})^k)$
and hence, by the observation in the beginning of this section,
\begin{equation}\label{eq:sravhvo}
\left|\sum_{k=k_0}^\infty \frac{R((\f_{i})^k)}k z^k\right|\le
\left|\sum_{k=k_0}^\infty \frac{R(\f^k)}k d^k\right|<\frac{\delta}{2}.
\end{equation}
Evidently,
\begin{equation}\label{eq:ozenkadlfinp}
\left|\sum_{k=k_0}^\infty \frac{R(\f^k)}k z^k\right|
\le \sum_{k=k_0}^\infty \frac{R(\f^k)}k d^k
<\frac{\delta}{2}.
\end{equation}
From (\ref{eq:sovpadcfin}), (\ref{eq:sravhvo}), and
(\ref{eq:ozenkadlfinp})  we obtain, for $i\ge i_0$,
$$
\left|u-\sum_{k=1}^\infty \frac{R((\f_{i})^k)}k z^k\right|\le
\left|\sum_{k=k_0}^\infty \frac{R(\f^k)}k z^k\right|+
\left|\sum_{k=k_0}^\infty \frac{R((\f_{i})^k)}k z^k\right|
<\frac{\delta}{2}+\frac{\delta}{2}=\delta.
$$
Thus, by the choice of $\delta$,
\begin{equation}\label{eq:areclose}
\|R_\f(z)-R_{\f_i}(z)\|<\e,\qquad i\ge i_0.
\end{equation}
Applying (\ref{eq:razzetaforfin}) to $\f_i$ with $i\ge i_0$
we deduce from (\ref{eq:areclose}) the estimation
$$
\left|R_\f(z)-\frac{1}{\det(1-B_i z)}\right|<\e.
$$
This completes the proof.
\end{proof}

By the definition of $R^\cC_\f(z)$ and Theorem \ref{teo:coinc_of_zetas}
we obtain the following
\begin{cor}\label{cor:asimfor}
The formula from Theorem \ref{teo:formulazeta}
remains valid for $R^\cC_\phi(z)$, $RT^f_\phi(z)$, and $RT^{ff}_\phi(z)$,
if $R(\phi^n)<\infty$ for all $n$.
\end{cor}

\section{Examples and counterexamples}\label{sec:examcounter}

In the remaining part of this section we will develop an
example from \cite{TroTwoEx}.

Let $F$ be a finite non-trivial group and $G=\oplus_{i\in\Z} F_i$,
$F_i\cong F$, i.e.
$$
\G=\{g=(\dots,g_{-1},g_{0},g_{1},g_2,\dots)\,|\, g_i \in F_i, \,
g_i\neq e \mbox{ only for a finite number of }i\}.
$$
Hence, $G$ is an infinitely generated residually
finite group.

Define $\phi$ to be the right shift: $\phi(g)_i=g_{i-1}$, $i\in\Z$.

\begin{lem}[{\cite[Lemma 3.1]{TroTwoEx}}]\label{lem:fromtwoex}
$R(\phi)=|F|<\infty$.
\end{lem}

\begin{proof}[Scetch of the proof]
For $a, g \in G$, the twisted conjugation has the form
$
(g a \phi(g^{-1}))_i=g_i a_i (g_{i-1})^{-1}.
$
For a twisted conjugation of two elements of the form
$$
\a:\quad\a_0=x,\quad \a_i=e\mbox{ for }i\neq 0, \qquad
\b:\quad\b_0=y,\quad \b_i=e\mbox{ for }i\neq 0,
$$
the above  formula gives
$$
g_0 x (g_{-1})^{-1}=y,\quad g_i (g_{i-1})^{-1}=e\mbox{ for }i\neq 0.
$$
So, $
g_0=g_1=\dots$, $g_{-1}=g_{-2}=\dots.
$
Since $g_i=e$ for large $i$, $g=(\dots,e,e,e,\dots)$. Thus, $\a$ and $\b$
are twisted conjugate if and only if they coincide.

Then we verify that any element $a=(\dots,a_i,\dots)$, $a_i=e$ for
$i<-m$ and $i>n$, is twisted conjugate to some element of the same form as $\a$ (with $x=a_{n}\dots a_{-m}$). 
\end{proof}

\begin{lem}[{\cite[Lemma 3.2]{TroTwoEx}}]\label{lem:fromtwoex2}
Suppose, $F$ has a trivial center. Then \emph{TBFT}$_f$ fails for $G$. 
\end{lem}

We will obtain below a more strong statement.

It is an easy well-known exercise to prove the following statement:
$$
\widehat{\oplus_i F_i}=\prod_i \widehat{F_i},  
$$
where $\prod$ means the topological (Tikhonoff) product.
Evidently, $\widehat{\phi}$ is the \emph{left} shift.
Its fixed points evidently are elements of $\prod$
with the same value $[\rho]$ for all $i$.
The corresponding representations $\rho^\infty$ are
finite-dimensional if and only if $\rho$ is one-dimensional.
If we denote by $\widehat{F}_{(1)}$ the subset of
$\widehat{F}$, formed by one-dimensional representations,
we obtain:
\begin{equation}\label{eq:sizeofRT}
RT(\phi)=|\widehat{F}|,\qquad RT^f(\phi)=|\widehat{F}_{(1)}|, 
\end{equation}
where we use $|.|$ for cardinality of a set.

Immediately from (\ref{eq:sizeofRT}) 
and Lemma \ref{lem:fromtwoex}
we obtain the following
statement.

\begin{teo}\label{teo:counterex}
If $F$ is not abelian, the above group $\oplus F$ is
an infinitely generated residually finite group, for which
neither TBFT, nor TBFT$_f$ is true. 
\end{teo}

Consider $\phi^n$ for this $\phi$. Then 
is the left shift by $n$ positions and $\oplus F$
decomposes in a sum of $n$ $\phi^n$-invariant
summands and at each of them $\phi^n$ acts in the
same way as $\phi$ on $\oplus F$, i.e. by the
left shift (by $1$ position).
Thus, we have the following explicit formulas:
\begin{equation}\label{eq:explex}
\begin{array}{c}
R(\phi^n)=(R(\phi))^n=|F|^n,\qquad RT(\phi^n)=(RT(\phi))^n=|\widehat{F}|^n,\\
 RT^f(\phi^n)=(RT^f(\phi))^n=|\widehat{F}_{(1)}|^n.
 \end{array}
\end{equation}

\begin{prop}\label{prop:calczetaex}
For the above group and automorphism $\phi$, zeta functions
$R_\phi(z)$, 
$R^\cU_\phi(z)=R^\cP_\phi(z)=RT^f_\phi(z)=RT^{ff}_\phi(z)$,
and $RT_\phi(z)$ are rational. More precisely,
$$
R_\phi(z)=\frac{1}{1-|F|z},\qquad
R^\cU_\phi(z)=\frac{1}{1-|\widehat{F}_{(1)}|z},
\qquad
RT_\phi(z)=\frac{1}{1-|\widehat{F}|z}.
$$
\end{prop}

\begin{proof}
Indeed, by (\ref{eq:explex}),
$$
R_\phi(z)=\exp\left(\sum_{n=1}^\infty \frac{(|F|z)^n}{n}\right)
=\exp(-\log(1-|F|z))=\frac{1}{1-|F|z}.
$$
Similarly for the remaining cases, using (\ref{eq:explex}). 
\end{proof}

\end{document}